\documentclass[12pt,english,reqno]{amsart}
\usepackage[T1]{fontenc}
\usepackage[latin9]{inputenc}
\usepackage[a4paper]{geometry}
\usepackage{mathtools}
\usepackage{amstext}
\usepackage{amsthm}
\usepackage{amssymb}
\usepackage{setspace}
\makeatletter
\numberwithin{equation}{section}
\numberwithin{figure}{section}
\theoremstyle{plain}
\newtheorem{thm}{\protect\theoremname}[section]
\theoremstyle{remark}
\newtheorem{rem}[thm]{\protect\remarkname}
\theoremstyle{plain}
\newtheorem{prop}[thm]{\protect\propositionname}
\theoremstyle{plain}
\newtheorem{lem}[thm]{\protect\lemmaname}
\theoremstyle{definition}
\newtheorem{defn}[thm]{\protect\definitionname}
\theoremstyle{remark}
\newtheorem{observation}[thm]{\protect\observationname}
\theoremstyle{remark}
\newtheorem{claim}[thm]{\protect\claimname}

\@ifundefined{date}{}{\date{}}
\usepackage{amsmath}
\usepackage{bbm}
\usepackage{dsfont}

\usepackage[unicode=true,pdfusetitle,
 bookmarks=true,bookmarksnumbered=false,bookmarksopen=false,
 breaklinks=false,pdfborder={0 0 0},pdfborderstyle={},backref=false,colorlinks=false]
 {hyperref}

\usepackage{tikz}

\usepackage{etoolbox}
\patchcmd{\@settitle}{\uppercasenonmath\@title}{}{}{}
\patchcmd{\@setauthors}{\MakeUppercase}{}{}{}
\patchcmd{\section}{\scshape}{}{}{}

\makeatother

\usepackage{babel}
\providecommand{\claimname}{Claim}
\providecommand{\definitionname}{Definition}
\providecommand{\lemmaname}{Lemma}
\providecommand{\observationname}{Observation}
\providecommand{\propositionname}{Proposition}
\providecommand{\remarkname}{Remark}
\providecommand{\theoremname}{Theorem}

\begin{document}
\global\long\def\One{\mathds{1}}%

\global\long\def\Laplacian{\Delta}%

\global\long\def\grad{\nabla}%

\global\long\def\norm#1{\left\Vert #1\right\Vert }%

\global\long\def\zz{\mathbb{Z}}%

\global\long\def\rr{\mathbb{R}}%

\global\long\def\nn{\mathbb{N}}%

\global\long\def\pp{\mathbb{P}}%

\global\long\def\ee{\mathbb{E}}%

\global\long\def\floor#1{\left\lfloor #1\right\rfloor }%

\global\long\def\ceil#1{\left\lceil #1\right\rceil }%

\global\long\def\var{\operatorname{Var}}%

\global\long\def\dd{\operatorname{d}}%

\global\long\def\gap{\operatorname{gap}}%

\global\long\def\meet{\operatorname{meet}}%

\title[FA$1$f in stationarity]{A note on the Fredrickson-Andersen one spin facilitated model in stationarity}
\author{Assaf Shapira}
\email{\href{mailto:assaf.shapira@normalesup.org}{assaf.shapira@normalesup.org}}
\urladdr{\href{https://assafshap.github.io/}{assafshap.github.io}}
\thanks{The author acknowledges the support of the ERC Starting Grant 680275
MALIG}
\begin{abstract}
This note discusses three problems related to the Fredrickson-Andersen
one spin facilitated model in stationarity. The first, considered
in 2008 in a paper of Cancrini, Martinelli, Roberto and Toninelli,
is the spectral gap of the model's infinitesimal generator. They study
the decay of this spectral gap when the density is large, but in dimensions
$3$ and higher, they do not find the exact exponent. They also show
that the persistence function of the model has exponential tail, but
the typical decay time is not analyzed. We will see that the correct
exponent for the decay of the spectral gap in dimension $3$ and higher
is $2$, and discover how the time over which the persistence function
decays diverges in high densities. We also discuss the scaling of
the spectral gap in finite graphs.
\end{abstract}

\maketitle

\section{Introduction and results}

The purpose of this note is to present three small results on the
Fredrickson-Andersen one spin facilitated model (FA1f), following
\cite{CMRT}, addressing two problems that have not been solved there
and one tightly related problem in a slightly different setting studied
in \cite{HMT,PillaiSmith,PillaiSmith2}. Since it is, in a sense,
an extension of \cite{CMRT}, the reader is referred to \cite{CMRT}
for the relevant background, references, and complete introduction
of the model and notation.

We will only briefly remind here that sites in $\zz^{d}$ could be
either occupied or empty, with equilibrium probabilities $1-q$ and
$q$ respectively (where $q$ is thought of as small). When a site
has at least one empty neighbor, it is being resampled from equilibrium
with rate $1$, and otherwise its occupation cannot change. The process
is reversible with respect to the invariant measure $\mu$, given
by an independent product of Bernoulli random variables with parameter
$1-q$. Probabilities and expected values with respect to the stochastic
process are denoted $\pp_{\mu}$ and $\ee_{\mu}$, where the subscript
$\mu$ indicates that the initial configuration is drawn from equilibrium.

The first result here completes Theorem 6.4 of \cite{CMRT}, which
bounds the spectral gap of the FA1f model. It is shown there that
the gap decays polynomially as the parameter $q$ tends to $0$, and
for dimensions $1$ and $2$ the exact exponent is identified, up
to a logarithmic correction in dimension $2$. For dimension $d\ge3$,
however, the exponent is bounded between $1+2/d$ and $2$, and its
exact value is not determined. The following theorem shows that the
correct scaling is $q^{2}$ --
\begin{thm}
\label{thm:gap}Consider the setting of \cite[Theorem 6.4]{CMRT},
in dimension $d\ge3$. Then there exists a positive constant $C$
(possibly depending on $d$) such that
\[
\gap(\mathcal{L})\le C\,q^{2}.
\]
\end{thm}

The second result presented here concerns with the persistence function.
Recall that 
\begin{align*}
F_{0}(t) & =\pp_{\mu}(\tau_{0}>t),\\
\tau_{0} & =\inf\{t:\text{origin is empty at time }t\}.
\end{align*}

In general, when the spectral gap is positive, $F_{0}(t)\le e^{-t/\overline{\tau}}$,
where $\overline{\tau}$ could be chosen to be equal $\frac{1}{\text{gap}\cdot q}$.
We will see that for the FA$1$f model on $\zz^{d}$ this choice is
not optimal, and that the typical time to empty the origin scales
(perhaps with lower order corrections) like the inverse of the spectral
gap, without the additional factor of $q$.
\begin{thm}
\label{thm:persistence}Consider the FA$1$f model. Then there exists
a positive constant $C$ such that 
\begin{alignat*}{2}
F_{0}(t) & \le e^{-Cq^{3}\,t}\qquad & d=1,\\
F_{0}(t) & \le e^{-C\frac{q^{2}}{\log(1/q)}\,t}\qquad & d=2,\\
F_{0}(t) & \le e^{-Cq^{2}\,t}\qquad & d\ge3.
\end{alignat*}
Moreover, for a different positive constant $C$, 
\begin{alignat*}{2}
\ee_{\mu}(\tau_{0}) & \ge Cq^{-3} & d=1,\\
\ee_{\mu}(\tau_{0}) & \ge Cq^{-2}\log(1/q)\qquad & d=2,\\
\ee_{\mu}(\tau_{0}) & \ge Cq^{-2} & d=3.
\end{alignat*}
\end{thm}

The last theorem that will be presented is in a slightly different
setting -- the Fredrickson-Andersen model on a finite graph $G$.
A particularly interesting case, studied in \cite{PillaiSmith,PillaiSmith2}
and more recently in \cite{HMT}, is when $\left|V(G)\right|=cq^{-1}$
for some positive constant $c$, where $V(G)$ denotes the set of
vertices of $G$. A lower bound on the spectral gap is given in \cite{PillaiSmith,PillaiSmith2}
and later on in \cite{HMT} by suggesting a relaxation mechanism in
which vacancies travel as random walkers on $G$. The next theorem
will bound the spectral gap of this model from above, showing that
this mechanism has a leading contribution. Consider two independent
continuous time random walks on $G$, namely, each of the two walkers
moves to each neighboring site with rate $1$. For two vertices $x,y\in V(G)$,
we define $\tau_{\meet}(x,y)$ to be the expected time that it takes
for two such random walkers starting at $x$ and $y$ to reach distance
at most $1$. Let $\overline{\tau}_{\meet}$ be its expected value,
starting at two random positions, i.e., $\overline{\tau}_{\meet}=\left|V(G)\right|^{-2}\sum_{x,y\in V(G)}\tau_{\meet}(x,y)$.
\begin{thm}
\label{thm:finitegraph}Consider the FA$1$f model on a finite graph
$G$, and assume that $\left|V(G)\right|=c/q$ for some positive constant
$c$. Let $\gap(\mathcal{L}_{G})$ denote its spectral gap with respect
to the product measure conditioned on having at least one vacancy.
Then
\[
\gap(\mathcal{L}_{G})\le\frac{Cq}{\overline{\tau}_{\meet}}
\]
for some positive constant $C$.
\end{thm}

\begin{rem}
\label{rem:finitegraph}\cite{HMT} give a lower bound on the spectral
gap, which in various graphs is of the same order as the upper bound
given in Theorem \ref{thm:finitegraph}. In particular, on the two
dimensional torus $\mathbb{T}^{2}=\zz^{2}/\ell\zz^{2}$, $\ell=cq^{-1/2}$,
both the upper and lower bounds scale like $q^{2}/\log(1/q)$. In
view of this scaling, and the relaxation mechanism reflected in its
proof, it seems that the correct scaling of the spectral gap is $q^{2}/\log(1/q)$
also in $\zz^{2}$, coinciding with the lower bound of \cite{CMRT}.
Unfortunately, the ideas in the proof of Theorem \ref{thm:finitegraph}
do not seem to be easily adapted for the model on $\zz^{2}$, and
the problem remains open.
\end{rem}

\section{Notation}

We will now recall some of the notation in \cite{CMRT} that will
be used in this note.
\begin{itemize}
\item For $\ell>0$, $\Lambda_{\ell}=\{0,\dots,\ell-1\}^{d}$ and $\zz^{d}(\ell)=\ell\zz^{d}$.
\item For $\ell>0$ and $x\in\zz^{d}(\ell)$ we denote $\Lambda_{x}=x+\Lambda_{\ell}$
(where the $\ell$-dependence of $\Lambda_{x}$ is implicit).
\item The configuration space is $\Omega=\{0,1\}^{\zz^{d}}$, and the measure
$\mu$ on this space is a product measure of Bernoulli random variables
with parameter $1-q$.
\item For $\eta\in\Omega$ and $x\in\zz^{d}$, the configuration which equals
$\eta$ outside $x$ and different from $\eta$ at $x$ is denoted
$\eta^{x}$.
\item The FA$1$f critical length is denoted $\ell_{q}=\left(\frac{\log(1-q_{0})}{\log(1-q)}\right)^{1/d}\approx Cq^{-1/d}$,
where $q_{0}\in(0,1)$ is given in \cite[Theorem 4.1]{CMRT}, and
does not depend on $q$.
\item The constraint of the FA$1$f dynamics, for a configuration $\eta\in\Omega=\{0,1\}^{\zz^{d}}$
and $x\in\zz^{d}$, is 
\[
c_{x}(\eta)=\begin{cases}
1 & \exists y\text{ such that }\norm{y-x}_{1}=1\text{ and }\eta(y)=0,\\
0 & \text{otherwise}.
\end{cases}
\]
\item The Dirichlet form of FA$1$f operating on a local function $f:\Omega\rightarrow\rr$
is given by
\begin{align*}
\mathcal{D}(f) & =\sum_{x\in\zz^{d}}\mu\left(c_{x}\var_{x}(f)\right)=q(1-q)\sum_{x\in\zz^{d}}\mu\left(c_{x}(f(\eta)-f(\eta^{x}))^{2}\right).
\end{align*}
\end{itemize}
For the FA$1$f model on a finite graph $G$ with vertex set $V(G)$
denote $\Omega_{G}=\{0,1\}^{V(G)}$; and $\mu_{G}$ the product measure
of Bernoulli random variables with parameter $1-q$, conditioned on
having at least one empty site. The constraint is defined in the same
manner as in $\zz^{d}$, except that $\norm{y-x}_{1}$ should be understood
as the graph distance between $x$ and $y$, that we denote $d(x,y)$.
The Dirichlet form operating on $f:\Omega_{G}\rightarrow\rr$ is given
by 
\[
\mathcal{D}_{G}(f)=q(1-q)\sum_{x\in V(G)}\mu_{G}\left(c_{x}(f(\eta)-f(\eta^{x}))^{2}\right).
\]

Throughout the proof $C$ will denote a generic positive constant,
and $q$ is assumed to be small enough.

\section{Proof of Theorem \ref{thm:gap}}

In order to bound the spectral gap from above, we need to find an
appropriate test function $f$, such that 
\[
\mathcal{D}(f)\le Cq^{2}\,\var(f).
\]

Consider the box $\Lambda=\Lambda_{\ell}$ for $\ell=\floor{1/q}$.
For a configuration $\eta$ and a site $x\in\Lambda$, the \emph{connected
cluster of $x$}, denoted $\mathcal{C}_{x}(\eta)$, is defined as
the set of sites $y\in\Lambda$ that are connected to $x$ via a path
of empty sites in $\Lambda$. If $\eta(x)=1$, its connected cluster
is the empty set. This way, the set of empty sites in $\Lambda$ is
partitioned in connected clusters, and we define: 
\begin{equation}
f(\eta)=\#\text{connected clusters in }\Lambda.\label{eq:gap_testfunction}
\end{equation}

\begin{prop}
For the test function $f$ defined in equation (\ref{eq:gap_testfunction}),
\begin{equation}
\var(f)\ge C\,q\,\ell^{d}.\label{eq:bigvariance}
\end{equation}
\end{prop}

\begin{proof}
This result is shown in \cite{CoxGrimmett} for the case of Bernoulli
bond percolation. We will repeat their argument applied to our case
for completeness.

First, note that we may write
\[
f(\eta)=\sum_{x\in\Lambda}\frac{1-\eta(x)}{|\mathcal{C}_{x}(\eta)|},
\]
where, when $\eta(x)=1$ (and therefore $\mathcal{C}_{x}(\eta)=\emptyset$),
we define $\frac{1-\eta(x)}{|\mathcal{C}_{x}(\eta)|}=0$.

Let $G=3\zz^{d}\cap\Lambda$, and for $A\subseteq G$, define $\chi_{A}(\eta)$
to be the indicator of the event, that the set $\{x\in\Lambda:\eta(y)=1\,\,\forall y\text{ such that }\norm{y-x}_{1}=1\}$
is equal $A$. Note that $\mu(\chi_{A})=(1-q)^{2d|A|}\cdot\left(1-(1-q)^{2d}\right)^{|G|-|A|}$.
For such a set $A$, let $D(A)=\{y\in\Lambda:\norm{y-x}_{1}=1\text{ for some }x\in A\}$,
and define 
\[
f_{A}(\eta)=\sum_{x\in\Lambda\setminus D(A)}\frac{1-\eta(x)}{\left|\mathcal{C}_{x}(\eta)\right|}.
\]

When $\eta$ is such that $\chi_{A}(\eta)=1$, 
\[
f(\eta)=f_{A}(\eta)+\sum_{x\in A}(1-\eta(x))\eqqcolon f_{A}(\eta)+n_{A}(\eta).
\]
In order to use this identity, we split the variance over the different
choices of $A$: 
\[
\var(f)=\sum_{A\subseteq G}\mu\left((f-\mu(f))^{2}\chi_{A}\right).
\]
Consider one of the summands in the above expression -- 
\begin{multline*}
\mu\left((f-\mu(f))^{2}\chi_{A}\right)=\mu\left((f_{A}-(\mu(f)+\mu(n_{A}))+n_{A}-\mu(n_{A}))^{2}\chi_{A}\right)\\
=\mu\left((f_{A}-(\mu(f)+\mu(n_{A}))^{2}\chi_{A}\right)+\mu\left((n_{A}-\mu(n_{A}))^{2}\chi_{A}\right)\\
+\mu\left((f_{A}-(\mu(f)+\mu(n_{A}))(n_{A}-\mu(n_{A}))\chi_{A}\right).
\end{multline*}
The first term is positive, and we will simply bound it by $0$. In
order to find the second term, we note that the variables $n_{A}$
and $\chi_{A}$ are independent, and therefore 
\[
\mu\left((n_{A}-\mu(n_{A}))^{2}\chi_{A}\right)=\mu(\chi_{A})\var(n_{A})=(1-q)^{2d|A|}\cdot\left(1-(1-q)^{2d}\right)^{|G|-|A|}\cdot\left|A\right|q(1-q).
\]
Finally, since under the event $\{\chi_{A}=1\}$ the variables $f_{A}$
and $n_{A}$ are independent, the third term vanishes. Therefore,
\[
\var(f)\ge\sum_{A\subseteq G}(1-q)^{2d|A|}\cdot\left(1-(1-q)^{2d}\right)^{|G|-|A|}\cdot\left|A\right|q(1-q)=q(1-q)^{2d+1}\left|G\right|.
\]
This establishes inequality (\ref{eq:bigvariance}).
\end{proof}
\begin{prop}
For the test function $f$ defined in equation (\ref{eq:gap_testfunction}),
\begin{equation}
\mathcal{D}(f)\le Cq^{3-d}.\label{eq:smalldirichlet}
\end{equation}
\end{prop}

\begin{proof}
Recall first that 
\[
\mathcal{D}(f)=q(1-q)\sum_{x\in\zz^{d}}\mu\left(c_{x}(\eta)\left(f(\eta^{x})-f(\eta)\right)^{2}\right).
\]
Consider a single term in that sum. First, note that by flipping a
single site $f$ could change by at most $2d$. If $x$ is outside
$\Lambda$, flipping it could not change the number of clusters in
$\Lambda$ and its contribution would be $0$. If $x$ is on the boundary
of $\Lambda$ (i.e., it is in $\Lambda$ and has a neighbor outside
$\Lambda$), then
\[
\mu\left(c_{x}(\eta)\left(f(\eta^{x})-f(\eta)\right)^{2}\right)\le2d\mu\left(c_{x}(\eta)\right)\le C\,q.
\]
Finally, if $x$ is in $\Lambda$ but has no neighbors outside $\Lambda$,
the number of open clusters could only change if it has at least two
empty neighbors -- 
\[
\mu\left(c_{x}(\eta)\left(f(\eta^{x})-f(\eta)\right)^{2}\right)\le2d\,\mu(\One_{x\text{ has at least }2\text{ empty neighbors}})\le C\,q^{2}.
\]
The proof is now concluded by summing these options -- 
\[
\sum_{x\in\zz^{d}}\mu\left(c_{x}(\eta)\left(f(\eta^{x})-f(\eta)\right)^{2}\right)\le C\ell^{d-1}q+C\ell^{d}q^{2}=Cq^{-d+2}.\qedhere
\]
\end{proof}
Theorem \ref{thm:gap} follows from equations (\ref{eq:bigvariance})
and (\ref{eq:smalldirichlet}), together with the variational characterization
of the spectral gap. \qed 

\section{Proof of Theorem \ref{thm:persistence}}

\subsection{Upper bound}

The basic tool for the proof of the upper bounds on $F_{0}(t)$ is
the following result of \cite{AdP} (see also \cite[Section 4]{RandomConstraint}):
\begin{lem}
\label{lem:AdP}Assume that, for some $\overline{\tau}>0$ and any
local function $f$ which vanishes on the event $\{\eta_{0}=0\}$,
\begin{equation}
\mu(f^{2})\le\overline{\tau}\mathcal{D}(f).\label{eq:AdP}
\end{equation}
Then $F_{0}(t)\le e^{-t/\overline{\tau}}$.
\end{lem}

We will use a path argument, similar to \cite[proof of Proposition 6.6]{CMRT},
proving inequality (\ref{eq:AdP}) with the appropriate $\overline{\tau}$.

We start by defining a canonical geometric path, which is a discrete
approximation of a straight segment. More precisely, for any $z\in\zz^{d}$,
we will construct a nearest neighbor path $\gamma(z)=(\gamma_{0}(z),\dots,\gamma_{n}(z))$
with $\gamma_{0}(z)=0$ and $\gamma_{n}(z)=z$ whose distance from
the line segment $[0,z]\in\rr^{d}$ is small. The exact definition
is rather cumbersome, and a reader who accepts that such a path could
be constructed satisfying Observation \ref{obs:canonicalpath_norm}
and Claim \ref{claim:conesize} (see Definition \ref{def:cone}) may
skip the technicalities involved in their proofs.
\begin{defn}
\label{def:geometricpath}Fix $z=(z_{1},\dots,z_{d})\in\zz^{d}$ with
$\norm z_{1}=n$. The \emph{canonical geometric path connecting $z$
to the origin} is the path $\gamma(z)=(\gamma_{0}(z),\dots,\gamma_{n}(z))$
constructed as follows -- consider the set $S\subseteq(0,1]\times\{1,\dots,d\}$
defined as 
\[
S=\left\{ (s,\alpha):sz_{\alpha}\in\zz\right\} .
\]
For each $\alpha$ there are $z_{\alpha}$ values of $s$ for which
$sz_{\alpha}\in\zz$, hence $|S|=n$. We will order $S$ according
to the lexicographic order, $(s_{1},\alpha_{1})<\dots<(s_{n},\alpha_{n})$,
so that $s_{i}\le s_{i+1}$ for all $i$, and in case of equality
$\alpha_{i}<\alpha_{i+1}$. Then 
\begin{align*}
\gamma_{0}(z) & =0,\\
\gamma_{i}(z) & =\gamma_{i-1}(z)+\vec{e}_{\alpha_{i}}\quad i\ge1.
\end{align*}
\end{defn}

\begin{observation}
\label{obs:canonicalpath_norm}Fix $z=(z_{1},\dots,z_{d})\in\zz^{d}$
and $0\le i\le\norm z_{1}$. Then $\norm{\gamma_{i}(z)}_{1}=i$, i.e.,
the sites of the path are indexed by their norm.
\end{observation}

\begin{claim}
\label{claim:geometricpath}Fix $z=(z_{1},\dots,z_{d})\in\zz^{d}$
with $\norm z_{1}=n$, and let $(s_{1},\alpha_{1})<\dots<(s_{n},\alpha_{n})$
be as in Definition \ref{def:geometricpath}. Then for all $1\le i\le n$,
\[
\gamma_{i}(z)=\floor{s_{i}z}^{{\scriptscriptstyle (\alpha_{i})}},
\]
where $\floor y^{{\scriptscriptstyle (\alpha)}}$, for $y=(y_{1},\dots,y_{d})\in\rr^{d}$,
is defined as 
\[
\floor y^{{\scriptscriptstyle (\alpha)}}=(\floor{y_{1}},\dots,\floor{y_{\alpha}},\ceil{y_{\alpha+1}}-1,\dots\ceil{y_{d}}-1).
\]
\end{claim}

\begin{proof}
We show this by induction. Start with $i=1$, and consider the vector
$s_{1}z$. By the construction of $s_{1}$, 
\begin{align*}
0<s_{1}z_{\alpha}<1 & \qquad\text{for }\alpha<\alpha_{1},\\
s_{1}z_{\alpha}=1 & \qquad\text{for }\alpha=\alpha_{1},\\
0<s_{1}z_{\alpha}\le1 & \qquad\text{for }\alpha>\alpha_{1};
\end{align*}
and indeed $\floor{s_{i}z}^{{\scriptscriptstyle (\alpha)}}=e_{\alpha}$.

For $i>1$, there are two options -- either $s_{i}=s_{i-1}$ and
$\alpha_{i}>\alpha_{i-1}$, or $s_{i}>s_{i-1}$. In the first case,
by induction and letting $y=s_{i-1}z$,
\begin{align*}
\gamma_{i}(z) & =\gamma_{i-1}(z)+e_{\alpha_{i}}\\
 & =(\floor{y_{1}},\dots,\floor{y_{\alpha_{i-1}}},\ceil{y_{\alpha_{i-1}+1}}-1,\dots,\ceil{y_{\alpha_{i}}},\dots,\ceil{y_{d}}-1).
\end{align*}
Since the coordinates between $\alpha_{i-1}$ and $\alpha_{i}$ are
not integer, we can replace $\ceil{\cdot}-1$ by $\floor{\cdot}$,
and since $y_{\alpha_{i}}$ is integer we may replace $\ceil{y_{\alpha_{i}}}$
by $\floor{y_{\alpha_{i}}}$. That is, $\gamma_{i}(z)=\floor y^{{\scriptscriptstyle (\alpha_{i})}}=\floor{s_{i}z}^{{\scriptscriptstyle (\alpha)}}$.

Let us now consider the second case, where $s_{i}>s_{i-1}$. First,
by induction, noting that the coordinates after $\alpha_{i-1}$ of
$s_{i-1}z$ are not integer, 
\[
\gamma_{i-1}(z)=(\floor{s_{i-1}z_{1}},\dots,\floor{s_{i-1}z_{d}}).
\]
On the other hand, we know that 
\begin{align*}
s_{i-1}z_{\alpha}<s_{i}z_{\alpha}<\floor{s_{i-1}z_{\alpha}}+1 & \qquad\text{for }\alpha<\alpha_{1},\\
s_{i}z_{\alpha}=\floor{s_{i-1}z_{\alpha}}+1 & \qquad\text{for }\alpha=\alpha_{1},\\
s_{i-1}z_{\alpha}<s_{i}z_{\alpha}\le\floor{s_{i-1}z_{\alpha}}+1 & \qquad\text{for }\alpha>\alpha_{1};
\end{align*}
so \textbf{$\floor{s_{i}z}^{{\scriptscriptstyle (\alpha)}}=\gamma_{i-1}(z)+\vec{e}_{\alpha}$}
and the proof is complete.
\end{proof}
\begin{defn}
\label{def:cone}Fix $\ell>0$ and $y\in\zz^{d}$ with $\norm y_{1}=m\le\ell$.
The the \emph{$\ell$-cone} of $y$ is the set 
\[
C_{y}^{(\ell)}=\{z\in\zz^{d}:m<\norm z_{1}\le\ell,\gamma_{m}(z)=y\}.
\]
\end{defn}

\begin{claim}
\label{claim:conesize}Fix $\ell>0$ and $y\in\Lambda_{\ell}$ such
that $\norm y_{1}\le\ell$. Then $|C_{y}^{(\ell)}|\le\frac{\ell^{d}}{\norm y_{1}^{d-1}+1}$.
\end{claim}

\begin{proof}
First, since for $y=0$ the cone $C_{y}^{(\ell)}$ consists of the
points in $\Lambda_{\ell}$ of norm smaller than $\ell$, its size
is smaller than $\ell^{d}$, so in what follows we may assume $y\neq0$;
and in this case we will show that the stronger inequality $|C_{y}|\le\frac{\ell^{d}}{\norm y_{1}^{d-1}}$
holds.

Let $z\in C_{y}$, i.e., $\gamma_{m}(z)=y$, so by Claim \ref{claim:geometricpath}
there exist $s$ and $\alpha$ such that $sz_{\alpha}=y_{\alpha}$
and 
\[
y=\floor{sz}^{\alpha}.
\]
Assume first $\alpha=1$, so in particular $y_{1}\neq0$ by the construction
of the geometric path. $s$ must be contained in $\{\frac{y_{1}}{k+y_{1}}\}_{k\in\nn}$;
so we fix $k$ and let $s=\frac{y_{1}}{k+y_{1}}$, such that
\[
z=\frac{k+y_{1}}{y_{1}}(y+\delta)
\]
for some $\delta\in\{0\}\times[0,1]^{d-1}$. That is, for all $\alpha>1$,
\[
z_{\alpha}\in(\frac{k}{y_{1}}+1)\,y_{\alpha}+[0,\frac{k}{y_{1}}+1],
\]
 allowing at most $(\frac{k}{y_{1}}+1)^{d-1}$ integer choices of
$z$. Finally, since $\norm z_{1}\le\ell$, necessarily $k\le k_{\text{max}}=(\frac{\ell}{\norm y_{1}}-1)y_{1}$,
so, still for $\alpha=1$, the number of possibilities for $z$ is
bounded by 
\[
(\frac{\ell}{\norm y_{1}}-1)y_{1}\cdot(\frac{k_{\text{max}}}{y_{1}}+1)^{d-1}\le y_{1}(\frac{\ell}{\norm y_{1}})^{d}.
\]
Finally, summing over all possible values of $\alpha$, 
\[
|C_{y}|\le\frac{\ell^{d}}{\norm y_{1}^{d}}\left(\sum_{\alpha=1}^{d}y_{\alpha}\right)=\frac{\ell^{d}}{\norm y_{1}^{d-1}}.\qedhere
\]
\end{proof}
For any $\ell>0$, let
\[
\chi_{\ell}(\eta)=\One_{\{\exists x\in\Lambda_{\ell},\,\eta(x)=0\}}.
\]

\begin{claim}
\label{claim:thepath}Fix $\eta\in\Omega$ and $\ell>0$ such that
$\chi_{\ell}=1$. Then there exists a path of configurations $\eta^{(0)},\dots,\eta^{(j)}$
and a sequence of sites $x_{0},\dots,x_{j-1}$ such that:
\begin{enumerate}
\item $\eta^{(0)}=\eta$ and $\eta_{0}^{(j)}=0$.
\item For any $i$, $\eta^{(i+1)}=\left(\eta^{(i)}\right)^{x_{i}}$, and
$c_{x_{i}}(\eta^{(i)})=1.$
\item The sites $x_{0},\dots,x_{j-1}$ all belong to the geometric path
$\gamma(x_{1})$. Moreover, each site of $\gamma(x_{1})$ appears
at most twice in the sequence $x_{0},\dots,x_{j-1}$, and in particular
$j\le2\ell$.
\item For all $i$, the number of sites in $\Lambda_{\ell}\setminus\{x_{i}\}$
which are empty for $\eta^{(i)}$ is at most the number of sites in
$\Lambda_{\ell}\setminus\{x_{i}\}$ which are empty for $\eta$.
\item Fix $z$, $\eta'$, $x'$. Then there exist at most one configuration
$\eta$ and one index $i$ such that $z=x_{1}$, $\eta'=\eta^{(i)}$
and $x'=x_{i}$. We write $(\eta,i)\sim(\eta',x',z)$.
\end{enumerate}
\end{claim}

\begin{proof}
The path is constructed in the same manner as \cite[proof of Proposition 6.6]{CMRT}
-- let $z$ be an empty site in $\Lambda_{\ell}$ with minimal $1$-norm,
and denote $\norm z_{1}=n$. Then set, for $i\in\{0,\dots,2n-2\}$,
\[
x_{i}=\begin{cases}
\gamma_{n-\frac{i}{2}-1}(z) & i\text{ even},\\
\gamma_{n-\frac{i-1}{2}}(z) & i\text{ odd}.
\end{cases}
\]
This sequence defines a path $\eta^{(0)},\dots,\eta^{(j)}$ that indeed
satisfied the conditions of the claim.
\end{proof}
\begin{prop}
\label{prop:pathargument}For any local function $f$ that vanishes
on the event $\{\eta_{0}=0\}$ and for any $\ell>0$,
\begin{align*}
\mu(\chi_{\ell}f^{2}) & \le\tau_{\ell}\mathcal{D}(f),\\
\tau_{\ell} & =\begin{cases}
C\ell^{2}q^{-1} & d=1,\\
C\log\ell\,\ell^{2}q^{-1} & d=2,\\
C\ell^{d}q^{-1} & d=3.
\end{cases}
\end{align*}
\end{prop}

\begin{proof}
First, since $f$ is local, we may restrict ourselves to proving the
inequality for FA$1$f on a large finite set $V\subset\zz^{d}$, so
the configuration space is $\Omega_{V}=\{0,1\}^{V}$. This allows
us to write the Dirichlet form as 
\begin{align*}
\mathcal{D}(f) & =\frac{1}{2}\sum_{\eta\in\Omega_{V}}\sum_{x\in V}R(\eta^{x},\eta)\left(f(\eta^{x})-f(\eta)\right)^{2},\\
R(\eta^{x},\eta) & =R(\eta,\eta^{x})=c_{x}(\eta)q(1-q)(\mu(\eta)+\mu(\eta^{x})).
\end{align*}

Consider, for any $\eta$, the path constructed in Claim \ref{claim:thepath},
and for $i\in\{0,\dots,j-1\}$ set 
\[
w_{i}=w(\norm{x_{i}}_{1})=(\norm{x_{i}}_{1}+1)^{(d-1)/2}.
\]
Note that we can bound, uniformly in $j$,
\begin{align*}
\sum_{i=1}^{j}w_{i}^{-2} & \le2\sum_{k=0}^{\ell}w(k)\le W,\\
W & =\begin{cases}
2\ell & d=1,\\
C\log\ell & d=2,\\
C & d\ge3;
\end{cases}
\end{align*}
and by Claim \ref{claim:conesize}, for every $y\in\Lambda_{\ell}$,
\begin{align*}
|C_{y}^{(d\ell)}|\,w(y)^{2} & \le C\ell^{d}.
\end{align*}
By the Cauchy-Schwarz inequality and the properties of the path, 
\begin{align}
\mu\left(\chi_{\ell}f^{2}\right) & =\mu\left(\chi_{\ell}\left(\sum_{i=1}^{j}\frac{1}{w_{i}}\,w_{i}(f(\eta^{(i)})-f(\eta^{(i-1)}))\right)^{2}\right)\label{eq:path_CS}\\
 & \le W\,\sum_{i}\mu\left(w_{i}^{2}c_{x_{i}}(\eta^{(i)})(f(\eta^{(i)})-f(\eta^{(i-1)}))^{2}\right)\nonumber \\
 & \le W\sum_{i=1}^{2\ell}\sum_{\eta\in\Omega_{V}}\mu(\eta)\sum_{\eta'\in\Omega_{V}}\sum_{x'\in V}\sum_{z\in C_{x'}^{(d\ell)}}\One_{\eta'=\eta_{i},\,x_{i}=x',\,z=x_{1}}w(x')^{2}c_{x'}(\eta')(f(\eta'^{x'})-f(\eta'))^{2}\nonumber \\
 & =W\sum_{\eta'}\sum_{x'}R(\eta'^{x'},\eta')\sum_{z\in C_{x'}^{(h\ell)}}\sum_{i}\sum_{\eta}\frac{\mu(\eta)}{R(\eta'^{x'},\eta')}\,\One_{(\eta,i)\sim(\eta',x',z)}w(x')^{2}c_{x'}(\eta')(f(\eta'^{x'})-f(\eta'))^{2}.\nonumber 
\end{align}
Note that we are allowed to divide by $R(\eta'^{x'},\eta')$ since
$c_{x'}(\eta')=1$, and hence it is non-zero. We can estimate $R(\eta'^{x'},\eta')$
more precisely: 
\[
R(\eta'^{x'},\eta')=q(1-q)\prod_{y\neq x'}\left((1-q)\eta'(y)+q(1-\eta'(y))\right),
\]
so 
\[
\frac{\mu(\eta)}{R(\eta'^{x'},\eta')}=\frac{(1-q)\eta(x')+q(1-\eta(x'))}{q(1-q)}\,\prod_{y\neq x'}\frac{(1-q)\eta(y)+q(1-\eta(y))}{(1-q)\eta'(y)+q(1-\eta'(y))}.
\]
By property 4 of the path we obtain
\[
\frac{\mu(\eta)}{R(\eta'^{x'},\eta')}\le q^{-1}.
\]
We now conclude by continuing the estimate (\ref{eq:path_CS}) --
\begin{align*}
\mu\left(\chi_{\ell}f^{2}\right) & \le C\ell^{d}Wq^{-1}\sum_{\eta'}\sum_{x'}R(\eta'^{x'},\eta')c_{x'}(\eta')(f(\eta'^{x'})-f(\eta'))^{2}=C\ell^{d}q^{-1}W\,\mathcal{D}(f).\qedhere
\end{align*}
\end{proof}
\begin{rem}
If, rather than $\mu(f^{2})$ in inequality (\ref{eq:AdP}), we would
like to bound $\mu(f)^{2}$, we could use Proposition \ref{prop:pathargument}
directly. By the Cauchy-Schwarz inequality 
\begin{align*}
\mu(f)^{2} & \le2\mu\left(\chi_{\ell}f\right)^{2}+2\mu\left((1-\chi_{\ell})f\right)^{2}\le2\mu\left(\chi_{\ell}f^{2}\right)+2\mu(1-\chi_{\ell})\mu(f^{2}).
\end{align*}
Choosing $\ell=C/q$ with $C$ small enough, $\mu(1-\chi_{\ell})$
is bounded below $1/4$, so
\[
\mu(f)^{2}\le4\mu\left(\chi_{\ell}f^{2}\right)+\var(f)\le C\tau_{\ell}\mathcal{D}(f).
\]
This inequality is not entirely worthless, and it does bound $\ee_{\mu}(\tau_{0})$
from above by $C\tau_{\ell}$ (see \cite[Section 4]{RandomConstraint}
and equation (\ref{eq:expected_tau0}) in the following section).
However, in order to obtain the exponential tail in Theorem \ref{thm:persistence}
a more sophisticated approach is required.
\end{rem}

From now on we set $\ell=\ell_{q}$. The lower bound on the spectral
gap of \cite[Theorem 6.4]{CMRT} is proven by introducing an auxiliary
dynamics with large spectral gap and then comparing it to the FA$1$f
dynamics. For that objective they define the constraints $\{\tilde{c}_{x}\}_{x\in\zz^{d}(\ell)}$,
stating that none of the boxes $\Lambda_{y}$ is entirely occupied
for $y\in x+\{\ell e_{1},\dots,\ell e_{d}\}$; with the associated
Dirichlet form 
\[
\tilde{\mathcal{D}}(f)=\sum_{x\in\zz^{d}(\ell)}\mu\left(\tilde{c}_{x}\var_{\Lambda_{x}}(f)\right).
\]

The following Lemma is given in \cite[equation (5.1)]{CMRT}:
\begin{lem}
\label{lem:auxiliarygap}The spectral gap associated with $\tilde{\mathcal{D}}$
is at least $1/2$. 
\end{lem}

Then, \cite[Theorem 6.4]{CMRT} is proved by showing:
\begin{prop}
\label{prop:dirichletcomparison}For any local function $f$, 
\[
\tilde{\mathcal{D}}(f)\le C\tau_{\ell}\mathcal{D}(f),
\]
where $\tau_{\ell}$ is defined in Proposition \ref{prop:pathargument}.
\end{prop}

We will use Lemma \ref{lem:auxiliarygap} in order to prove the following
claim:
\begin{claim}
\label{claim:hittingtimeandgap}Assume that $q$ is small enough,
and let $g$ be a function vanishing on the event $\{\chi_{\ell}=1\}$.
Then 
\[
\mu(g^{2})\le C\tilde{\mathcal{D}}(g).
\]
\end{claim}

\begin{proof}
By Lemma \ref{lem:auxiliarygap}, the spectral gap of the dynamics
described by the Dirichlet form $\tilde{\mathcal{D}}$ is at least
$1/2$. A simple application of Chebyshev's inequality (see \cite[Claim 4.11]{RandomConstraint})
then yields 
\[
\mu(g^{2})\le\frac{1+\mu(\chi_{\ell})}{\mu(\chi_{\ell})}\,2\mathcal{\tilde{D}}(g),
\]
and the result follows since $\ell=\ell_{q}$, hence $\mu(\chi_{\ell})$
is bounded away from $0$.
\end{proof}
We are now ready to prove the upper bound on the persistence function
-- consider $f$ which vanishes on $\{\eta_{0}=0\}$. Then, by Claim
\ref{claim:hittingtimeandgap} and the fact that $\chi_{\ell}$ does
not depend on the occupation in $\Lambda_{x}$ for $x\neq0$,
\begin{align*}
\mu\left((1-\chi_{\ell})f^{2}\right) & \le C\tilde{\mathcal{D}}\left((1-\chi_{\ell})f\right)=C\sum_{x\in\zz^{d}}\mu\left(\tilde{c}_{x}\var_{\Lambda_{x}}\left((1-\chi_{\ell})f\right)\right)\\
 & =C\sum_{x\in\zz^{d}\setminus\{0\}}\mu\left(\tilde{c}_{x}\var_{\Lambda_{x}}\left((1-\chi_{\ell})f\right)\right)+C\mu\left(\tilde{c}_{0}\var_{\Lambda_{0}}\left((1-\chi_{\ell})f\right)\right)\\
 & \le C\sum_{x\in\zz^{d}\setminus\{0\}}\mu\left(\tilde{c}_{x}\var_{\Lambda_{x}}(f)\right)+C\sum_{x\in\Lambda_{0}}\mu\left(\tilde{c}_{0}\var_{x}\left((1-\chi_{\ell})f\right)\right)\\
 & \le C\tilde{\mathcal{D}}(f)+C\sum_{x\in\Lambda_{0}}\mu\left(\tilde{c}_{0}\var_{x}\left((1-\chi_{\ell})f\right)\right).
\end{align*}

The first term could be bounded using Proposition \ref{prop:dirichletcomparison}.
In order to bound the second term, we note that when $\chi_{\ell}(\eta^{x})\neq\chi_{\ell}(\eta)$,
necessarily $\chi_{\ell}(\eta^{1\leftarrow x})=0$, where $\eta^{1\leftarrow x}$
is the configuration that equals $\eta$ outside $x$ and $1$ at
$x$. Therefore,
\begin{align*}
\var_{x}\left((1-\chi_{\ell})f\right) & \le q(1-q)\left((1-\chi_{\ell}(\eta^{x}))f(\eta^{x})-(1-\chi_{\ell}(\eta))f(\eta)\right)^{2}\\
 & =q(1-q)\One_{\chi_{\ell}(\eta^{x})\neq\chi_{\ell}(\eta)}\left(f(\eta^{1\leftarrow x})\right)^{2}\\
 & \le q\left((1-q)\left(f(\eta^{1\leftarrow x})\right)^{2}+q\left(f(\eta^{0\leftarrow x})\right)^{2}\right)\\
 & =q\mu_{x}(f^{2}).
\end{align*}
Since $\tilde{c}_{0}$ does not depend on the occupation in $\Lambda_{0}$,
\[
\sum_{x\in\Lambda_{0}}\mu\left(\tilde{c}_{0}\var_{x}\left((1-\chi_{\ell})f\right)\right)\le\sum_{x\in\Lambda_{0}}q\mu\left(\tilde{c}_{0}f^{2}\right)\le\sum_{x\in\Lambda_{0}}q\mu\left(\chi_{2\ell}f^{2}\right)\le C\mu\left(\chi_{2\ell}f^{2}\right),
\]
which could be bounded using Proposition \ref{prop:pathargument}.

We have so far shown that
\[
\mu\left((1-\chi_{\ell})f^{2}\right)\le C\tau_{\ell}\mathcal{D}(f)+C\tau_{2\ell}\mathcal{D}(f)\le C\tau_{\ell}\mathcal{D}(f).
\]
Using Proposition \ref{prop:pathargument} again $\mu\left(\chi_{\ell}f^{2}\right)\le\tau_{\ell}\mathcal{D}(f)$,
and therefore
\[
\mu(f^{2})=\mu\left((1-\chi_{\ell})f^{2}\right)+\mu\left(\chi_{\ell}f^{2}\right)\le C\tau_{\ell}\mathcal{D}(f),
\]
i.e., inequality (\ref{eq:AdP}) holds with $\overline{\tau}=C\,\tau_{\ell}$;
and Lemma \ref{lem:AdP} concludes the proof of the upper bound.\qed

\subsection{Lower bound}

In order to bound the expected value of $\tau_{0}$ from below, we
will use the following variational principle (see \cite[Proposition 4.7]{RandomConstraint}):
\begin{lem}
Let $V_{0}$ be the space of local functions that vanish on the event
$\{\eta(0)=0\}$. Then 
\[
\ee_{\mu}(\tau_{0})=\sup_{f\in V_{0}}(2\mu(f)-\mathcal{D}(f)).
\]
\end{lem}

\begin{rem}
It would be more convenient to use a homogeneous version of that variational
principle -- 
\begin{align*}
\ee_{\mu}(\tau_{0}) & =\sup_{f\in V_{0}}\sup_{\lambda\in\rr}(2\mu(\lambda f)-\mathcal{D}(\lambda f))=\sup_{f\in V_{0}}\sup_{\lambda\in\rr}(2\lambda\mu(f)-\lambda^{2}\mathcal{D}(f)),
\end{align*}
and since for fixed $f$ the expression is maximized for $\lambda=\frac{\mu(f)}{\mathcal{D}(f)}$,
\begin{equation}
\ee_{\mu}(\tau_{0})=\sup_{f\in V_{0}}\frac{\mu(f)^{2}}{\mathcal{D}(f)}.\label{eq:expected_tau0}
\end{equation}
\end{rem}

We will now treat separately three different cases: $d\ge3$, $d=1$,
and $d=2$.

\subsubsection{$d\ge3$}

For high dimensions, we will use the test function $f(\eta)=\eta_{0}.$
Its expected value is $1-q$, and its Dirichlet form is given by $\mathcal{D}(f)=q(1-q)\mu(c_{0})\le2q^{2}.$
Equation (\ref{eq:expected_tau0}) now concludes the proof of this
case. \qed

\subsubsection{$d=1$}

In the one dimensional case we will use a test function similar to
\cite[proof of Theorem 6.4]{CMRT}. Let $\ell=\ceil{1/q}$,
\[
\xi(\eta)=\inf\{|x|:\eta_{x}=0\},
\]
 and 
\begin{equation}
f(\eta)=\xi\One_{\xi<\ell}+(2\ell-\xi)\One_{\ell\le\xi<2\ell}.\label{eq:testfunction_time1d}
\end{equation}

\begin{prop}
\label{prop:bigmean_time1d}Consider $f$ defined in equation (\ref{eq:testfunction_time1d}).
Then 
\[
\mu(f)\ge C\ell.
\]
\end{prop}

\begin{proof}
First, note that $\xi$ is a geometric random variable with parameter
$1-(1-q)^{2}$, so we can calculate explicitly 
\[
\mu(f>\ell/2)=\mu(\frac{\ell}{2}<\xi<\frac{3\ell}{2})=(1-q)^{2\,\ell/2}(1-(1-q)^{2\,\ell})>C,
\]
and since $f$ is positive $\mu(f)\ge C\ell$.
\end{proof}
\begin{prop}
Consider $f$ defined in equation (\ref{eq:testfunction_time1d}).
Then 
\[
\mathcal{D}(f)\le4q.
\]
\end{prop}

\begin{proof}
In order to bound $\mathcal{D}(f)$ we make the following observations
--
\begin{enumerate}
\item For fixed $\eta$, if $f(\eta)\neq f(\eta^{x})$ then either $\xi(\eta)=|x|$,
or $\xi(\eta)=|x|+1$.
\item For fixed $\eta$, if $f(\eta)\neq f(\eta^{x})$ then $\left(f(\eta)-f(\eta^{x})\right)^{2}=1$.
\end{enumerate}
With these observations in mind,
\begin{align*}
\mathcal{D}(f) & =q(1-q)\sum_{x}\mu\left(c_{x}(f(\eta^{x})-f(\eta))^{2}\right)\\
 & \le q(1-q)\sum_{x}\mu\left(c_{x}(\One_{\xi=|x|}+\One_{\xi=|x|+1})\right)\\
 & \le4q(1-q)\qedhere
\end{align*}
\end{proof}
Using these two propositions and equation (\ref{eq:expected_tau0}),
the case $d=1$ is concluded.\qed

\subsubsection{$d=2$}

Let $\ell=\ell_{q}$, recalling $\ell_{q}=Cq^{-1/2}$, and $\Lambda=\{x\in\zz^{2}:\norm x_{1}\le\ell\}$.
The test function we will use is 
\begin{align}
f(\eta) & =\inf_{\substack{x\in\zz^{2}\\
\eta(x)=0
}
}\log(1+\norm x_{1}\wedge\ell).\label{eq:testfunction_time2d}
\end{align}
Note that it vanishes on the event $\{\eta(0)=0\}$, and that it depends
only on the occupation in $\Lambda$.
\begin{rem}
\label{rem:SaloffCoste}The function $\log(1+\norm x_{1})$ is used
in \cite{SaloffCoste} in a different context, in order to bound the
relaxation time of the simple random walk on a certain graph that
consists of two copies of $\Lambda_{n}$ (for some $n\in\nn$). Though
presented differently, the proof there is based on the fact that this
function serves as a test function for the hitting time at $0$ of
the random walk on $\Lambda_{n}$; and that for the dynamics to relax
a random walk in one of the two copies of $\Lambda_{n}$ must first
hit $0$. Indeed, the bound $Cq^{-2}\log(1/q)$ obtained scales as
the expected hitting time at the origin for a random walk in $\Lambda$
with jump rate $q$.
\end{rem}

\begin{prop}
Consider $f$ defined in equation (\ref{eq:testfunction_time2d}).
Then 
\[
\mu(f)\ge C\log(1/q).
\]
\end{prop}

\begin{proof}
The proof is based on the fact, that the probability that $\Lambda$
is entirely occupied, given by $(1-q)^{\left|\Lambda\right|}$, is
bounded away from $0$ uniformly in $q$ (thanks to the choice $\ell=Cq^{-1/2}$).
In this case, $f$ equals $\log(1+\ell)$, which is greater than $\frac{1}{2}\log(1/q)$.
\end{proof}
\begin{prop}
Consider $f$ defined in equation (\ref{eq:testfunction_time2d}).
Then 
\[
\mathcal{D}(f)\le Cq^{2}\log(1/q).
\]
\end{prop}

\begin{proof}
The proof is based on the following observation:
\begin{observation}
Fix $\eta\in\Omega$ and $x\in\Lambda$ such that $c_{x}(\eta)=1$,
$\eta(x)=0$, and $f(\eta)\neq f(\eta^{x})$. Then $f(\eta)=\log(1+\norm x_{1})$
and $f(\eta^{x})=\log(2+\norm x_{1})$.
\end{observation}

\begin{proof}
Since $f(\eta)\neq f(\eta^{x})$, there can be no vertex $y$ with
$\eta(y)=0$ and $\norm y_{1}\le\norm x_{1}$, so in particular $f(\eta)=\log(1+\norm x_{1})$.
Moreover, since $c_{x}=1$, it must have an empty neighbor $z$, and
since this neighbor has norm greater than $\norm x_{1}$, necessarily
$\norm z_{1}=\norm x_{1}+1$. Since, in addition, no empty site for
$\eta^{x}$ has norm strictly smaller than $\norm x_{1}+1$, we conclude
that $f(\eta^{x})=\log(2+\norm x_{1})$.
\end{proof}
This observation implies in particular that for all $\eta\in\Omega$
and $x\in\Lambda$ such that $c_{x}(\eta)=1$ 
\[
\left(f(\eta^{x})-f(\eta)\right)^{2}\le\left(1+\norm x_{1}\right)^{-2}.
\]

Using this estimate, 
\begin{align*}
\mathcal{D}(f) & =q(1-q)\sum_{x}\mu\left(c_{x}(f(\eta^{x})-f(\eta))^{2}\right)\\
 & \le q(1-q)\sum_{x\in\Lambda}\mu\left(c_{x}\right)\left(1+\norm x_{1}\right)^{-2}\\
 & \le Cq^{2}\log(\ell)=Cq^{2}\log(1/q)\qedhere
\end{align*}
\end{proof}
The proof of Theorem \ref{thm:persistence} is then concluded by the
last two propositions and equation (\ref{eq:expected_tau0}).\qed

\section{Proof of Theorem \ref{thm:finitegraph}}

As in the proof of Theorem \ref{thm:gap}, we look for $f:\Omega_{G}\rightarrow\rr$
such that
\[
\mathcal{D}_{G}(f)\le Cd_{\max}q\overline{\tau}_{\meet}^{-1}\,\var(f),
\]
where the variance is understood with respect to the measure $\mu_{G}$.

The test function that we use is 
\begin{equation}
f(\eta)=\max_{\substack{x,y\in V(G)\\
\eta(x)=\eta(y)=0
}
}\tau_{\text{meet}}(x,y).\label{eq:testfunction_finite}
\end{equation}

Before analyzing the function $f$, note that $\tau_{\text{meet}}$
solves the following Poisson problem: 
\begin{align*}
-\mathcal{L}_{\text{RW}}\left(\tau_{\text{meet}}(x,y)\right) & =1,\qquad d(x,y)>1,\\
\tau_{\text{meet}}(x,y) & =0,\qquad d(x,y)\le1;
\end{align*}
where $\mathcal{L}_{\text{RW}}$ is the infinitesimal generator of
two independent random walks on $G$. Multiplying both sides by $\tau_{\text{meet}}(x,y)$
and averaging over $x$ and $y$ we obtain 
\begin{equation}
\overline{\tau}_{\text{meet}}=\mathcal{D}_{\text{RW}}\left(\tau_{\text{meet}}\right),\label{eq:dirichlet_taumeet}
\end{equation}
where the Dirichlet form is given for every $g:V(G)\times V(G)\rightarrow\rr$
by 
\[
\mathcal{D}_{\text{RW}}\left(g\right)=\frac{1}{2\left|V(G)\right|^{2}}\sum_{x}\sum_{y}\left[\sum_{x'\sim x}\left(g(x',y)-g(x,y)\right)^{2}+\sum_{y'\sim y}\left(g(x,y')-g(x,y)\right)^{2}\right].
\]

\begin{prop}
\label{prop:finite_var}For $f$ defined in equation (\ref{eq:testfunction_finite}),
\[
\var(f)\ge C\overline{\tau}_{\meet}^{2}.
\]
\end{prop}

\begin{proof}
Under $\mu_{G},$ the probability that exactly one site is empty is
of order $1$ (i.e., bounded away from $0$ uniformly in $q$). When
this happens, $f=0$, so 
\begin{align*}
\var(f) & =\mu_{G}\left[\left(f-\mu_{G}(f)\right)^{2}\right]\ge\mu_{G}\left[\left(f-\mu_{G}(f)\right)^{2}\One_{\sum_{x}(1-\eta(x))=1}\right]\\
 & =\mu_{G}(f)^{2}\,\mu_{G}\left[\sum_{x}(1-\eta(x))=1\right]\ge C\mu_{G}(f)^{2}.
\end{align*}
Moreover, the probability that there are exactly two vacancies is
also of order $1$. Under this event, 
\[
\mu_{G}\left[f(\eta)\middle|\One_{\sum_{x}(1-\eta(x))=2}\right]=\frac{1}{\left|V(G)\right|(\left|V(G)\right|-1)}\sum_{\substack{x,y\in V(G)\\
x\neq y
}
}\tau_{\text{meet}}(x,y)\ge\overline{\tau}_{\text{meet}}.\qedhere
\]
\end{proof}
\begin{prop}
\label{prop:finite_dirichlet}For $f$ defined in equation (\ref{eq:testfunction_finite}),
\[
\mathcal{D}(f)\le Cq\overline{\tau}_{\meet}.
\]
\end{prop}

\begin{proof}
We start the proof with an observation:
\begin{observation}
Let $\eta\in\Omega_{G}$ and $x\in V(G)$ such that $c_{x}(\eta)=1$,
$\eta(x)=0$, and $f(\eta)\neq f(\eta^{x})$. Then there exist $x',y\in V(G)$
such that $x'\sim x$, $d(x,y)>1$, $\eta(y)=\eta(x')=0$, and 
\[
\tau_{\meet}(x',y)\le f(\eta^{x})<f(\eta)=\tau_{\meet}(x,y).
\]
\end{observation}

\begin{proof}
First, recalling equation (\ref{eq:testfunction_finite}), when filling
an empty site $f$ could only decrease, and since $f(\eta)\neq f(\eta^{x})$
necessarily $f(\eta)>f(\eta^{x})$. Moreover, $f$ could only change
if the maximum is attained at the pair $x,y$ for some $y\in V(G)$,
i.e., $f(\eta)=\tau_{\text{meet}}(x,y)$. Note that $f(\eta)$ is
non-zero, hence $d(x,y)>1$. Finally, $c_{x}(\eta)=1$ means that
$x$ has an empty neighbor $x'$; and since in the configuration $\eta^{x}$
both $x'$ and $y$ are empty $f(\eta^{x})\ge\tau_{\text{meet}}(x',y)$.
\end{proof}
As a consequence of this observation, for all $\eta\in\Omega_{G}$
and $x$ such that $c_{x}(\eta)=1$, 
\[
\left(f(\eta^{x})-f(\eta)\right)^{2}\le\sum_{\substack{y\in V(G)\\
d(x,y)>1
}
}\sum_{\substack{x'\in V(G)\\
x'\sim x
}
}(1-\eta(y))(1-\eta(x'))\,\left(\tau_{\text{meet}}(x,y)-\tau_{\text{meet}}(x',y)\right)^{2}.
\]

We can now use this estimate and calculate the Dirichlet form: 
\begin{align*}
\mathcal{D}(f) & =q(1-q)\sum_{x}\mu_{G}\left[c_{x}\left(f(\eta^{x})-f(\eta)\right)^{2}\right]\\
 & \le q(1-q)\sum_{x}\mu_{G}\left[\sum_{\substack{y\in V(G)\\
d(x,y)>1
}
}\sum_{x'\sim x}(1-\eta(y))(1-\eta(x'))\,\left(\tau_{\text{meet}}(x,y)-\tau_{\text{meet}}(x',y)\right)^{2}\right]\\
 & \le q^{3}(1-q)\sum_{x}\sum_{\substack{y\in V(G)\\
d(x,y)>1
}
}\sum_{x'\sim x}\left(\tau_{\text{meet}}(x,y)-\tau_{\text{meet}}(x',y)\right)^{2}\\
 & \le Cq\,\mathcal{D}_{\text{RW}}(\tau_{\text{meet}}),
\end{align*}
and the proposition follows form equation (\ref{eq:dirichlet_taumeet}).
\end{proof}
Theorem \ref{thm:finitegraph} is a consequence of Propositions \ref{prop:finite_var}
and \ref{prop:finite_dirichlet}, using the variational characterization
of the spectral gap. \qed

\begin{rem}
For the case of the two dimensional torus discussed in Remark \ref{rem:finitegraph},
it is possible to bound $\overline{\tau}_{\text{meet}}$ from below
using the equivalent of equation (\ref{eq:expected_tau0}): 
\[
\overline{\tau}_{\meet}=\sup_{g}\frac{\left(\frac{1}{\left|V(G)\right|^{2}}\sum_{x,y}g(x,y)\right)^{2}}{\mathcal{D}_{\text{RW}}(g)},
\]
where the supremum is taken over functions $g:V(G)\times V(G)\rightarrow\rr$,
for which $g(x,y)=0$ whenever $d(x,y)\le1$. Taking the test function
$g(x,y)=\log(d(x,y)\lor1)$ (cf. Remark \ref{rem:SaloffCoste}) yields
the bound $q^{2}/\log(1/q)$ on the spectral gap. 
\end{rem}

\section*{Acknowledgments}

I wish to thank Ivailo Hartarsky and Fabio Martinelli for the discussions
and comments.

\bibliographystyle{amsplain}
\bibliography{fa1f_stationarity}

\vspace{2cm}
\end{document}